\documentclass[a4paper , 11pt]{amsart}
\usepackage{amsmath}
\usepackage{amsthm}
\usepackage{amsfonts}
\newtheorem{thm}{Theorem}[section]
\newtheorem{lem}[thm]{Lemma}
\newtheorem{prop}[thm]{Proposition}
\newtheorem{cor}[thm]{Corollary}
\theoremstyle{definition}
\newtheorem*{rem}{Remark}

\title{Transcendental $p$-adic continued fractions}
\author{Tomohiro Ooto}
\date{}
\begin{document}
\address{Graduate School of Pure and Applied Sciences, University of Tsukuba, Tennodai 1-1-1, Tsukuba, Ibaraki, 305-8571, Japan}
\email{ooto@math.tsukuba.ac.jp}

\maketitle
\begin{abstract}
We establish a new transcendence criterion of $p$-adic continued fractions which are called Ruban continued fractions.
By this result, we give explicit transcendental Ruban continued fractions with bounded $p$-adic absolute value of partial quotients.
This is $p$-adic analogy of Baker's result.
We also prove that $p$-adic analogy of Lagrange Theorem for Ruban continued fractions is not true.
\end{abstract}


\section{Introduction}
Maillet \cite{Maillet} is the first person who gave explicit transcendental continued fractions with bounded partial quotients.
After that, Baker \cite{Baker} extended Maillet's results with LeVeque Theorem \cite{LeVeque} which is Roth Theorem for algebraic number fields.

There exist continued fraction expansions in $p$-adic number field $\mathbb{Q}_p$, not just in $\mathbb{R}$.
Schneider \cite{Schneider} was motivated by Mahler's work \cite{Mahler} and gave an algorithm of $p$-adic continued fraction expansion.
In the same year, Ruban \cite{Ruban} also gave an different algorithm of $p$-adic continued fraction expansion.
Ubolsri, Laohakosol, Deze, and Wang gave several transcendence criteria for Ruban continued fractions (see \cite{Laohakosol2,Wang3,Wang1,Wang2}).
The proofs are mainly based on the theory of $p$-adic Diophantine approximations.
However, they only studied Ruban continued fractions with unbounded $p$-adic absolute value of partial quotients.
In this paper, we study analogy of Baker's transcendence criterion for Ruban continued fractions with bounded $p$-adic absolute value of partial quotients.

Let $p$ be a prime.
We denote by $|\cdot |_p$ the valuation normalized to satisfy $|p|_p =1/p$. 
A function $\lfloor \cdot \rfloor _p$ is given by the following:
$$ \lfloor \cdot \rfloor _p : \mathbb{Q}_p \rightarrow \mathbb{Q} \ ;
\lfloor \alpha \rfloor _p
= \begin{cases}
\sum_{n=m}^{0} c_n p^n & (m \leq 0), \\
0  & (m>0),
\end{cases} $$
where $\alpha = \sum_{n=m}^{\infty } c_n p^n ,\ c_n \in \{0,1, \ldots , p-1 \},\ m \in \mathbb{Z},\ c_m \not= 0 .$
The function is called a {\itshape $p$-adic floor function}.
If $\alpha \not = \lfloor \alpha \rfloor _p$, then we can write $\alpha $ in the form
$$ \alpha = \lfloor \alpha \rfloor _p + \frac{1}{\alpha _1}$$
with $\alpha _1 \in \mathbb{Q}_p$. Note that $ |\alpha _1|_p \geq p$ and $ \lfloor \alpha _1 \rfloor _p \not = 0$.
Similarly, if $\alpha _1 \not = \lfloor \alpha _1 \rfloor _p$, then we have
$$ \alpha _1 = \lfloor \alpha _1 \rfloor _p + \frac{1}{\alpha _2}$$
with $\alpha _2 \in \mathbb{Q}_p$.
We continue the above process provided $\alpha _n \not = \lfloor \alpha _n \rfloor _p$.
In this way, it follows that $\alpha $ can be written in the form
$$ \alpha = \lfloor \alpha \rfloor _p +\cfrac{1}{ \lfloor \alpha _1 \rfloor _p +\cfrac{1}{ \lfloor \alpha _2 \rfloor _p +\cfrac{1}{\ddots \lfloor \alpha _{n-1} \rfloor _p + \cfrac{1}{\alpha _n }}}}.$$
For simplicity of notation, we write the continued fraction 
$$ [ \lfloor \alpha \rfloor _p ,\lfloor \alpha _1 \rfloor _p ,\lfloor \alpha _2 \rfloor _p , \ldots , \lfloor \alpha _{n-1} \rfloor _p , \alpha _n ].$$
$\alpha _n$ is called the {\itshape $n$-th complete quotient} and either $\lfloor \alpha \rfloor _p$ or $\lfloor \alpha _n \rfloor _p$  is called a {\itshape partial quotient}.
If the above process stopped in a certain step, then
$$ \alpha =[ \lfloor \alpha \rfloor _p ,\lfloor \alpha _1 \rfloor _p ,\lfloor \alpha _2 \rfloor _p , \ldots , \lfloor \alpha _{n-1} \rfloor _p , \lfloor \alpha _n \rfloor _p ]$$
is called a {\itshape finite Ruban continued fraction}.
Otherwise, in the same way, we have
$$ \alpha =[ \lfloor \alpha \rfloor _p ,\lfloor \alpha _1 \rfloor _p ,\lfloor \alpha _2 \rfloor _p , \ldots , \lfloor \alpha _{n-1} \rfloor _p , \lfloor \alpha _n \rfloor _p , \ldots ]$$
which is called an {\itshape infinite Ruban continued fraction}.
As a remark, according to the fact that the Hensel expansion of a $p$-adic number is unique, we have the uniqueness of Ruban continued fraction expansions.

We define $S_p = \{ \lfloor \alpha \rfloor _p \ | \ \alpha \in \mathbb{Q}_p \},\ S' _p =\{\lfloor \alpha \rfloor _p \ | \ |\alpha |_p \geq p \ \mbox{for}\ \alpha \in \mathbb{Q}_p \}$.
Let $(a_i)_{i\geq 0}$ be a sequence with $a_0 \in S_p$ and $a_i \in S' _p$ for all $i\geq 1$, and $(n_i)_{i\geq 0}$ be an increasing sequence of positive integers.
Let $(\lambda _i)_{i\geq 0}$ and $(k_i)_{i\geq 0}$ be sequences of positive integers.
Assume that for all $i$, 
\begin{gather*}
n_{i+1} \geq n_i +\lambda _i k_i \\
a_{m+k_i} = a_m \mbox{ for } n_i \leq m \leq n_i +(\lambda _i -1)k_i -1. 
\end{gather*}
Consider a $p$-adic number $\alpha $ defined by
$$\alpha =[a_0 , a_1 , a_2 , \ldots , a_n , \ldots ].$$
Then $\alpha $ is called a {\itshape quasi-periodic Ruban continued fraction}.

The main theorem is the following.
\begin{thm} \label{Main Thm}
Let $(a_i)_{i\geq 0}, (n_i)_{i\geq 0}, (\lambda _i)_{i\geq 0}, \mbox{ and } (k_i)_{i\geq 0}$ be as in the above, and $A\geq p$ be a real number.
Assume that $(a_i)_{i\geq 0}$ is a non-ultimately periodic sequence  such that $|a_i |_p \leq A$ for each $i$.
If $a_{n_i} = a_{n_i +1} = \cdots = a_{n_i + k_i -1} = (p-1)+(p-1)p^{-1}=p-p^{-1}$ for infinitely many $i$ and
$$\liminf_{i\rightarrow \infty } \frac{\lambda _i}{n_i} > B=B(A),$$
where $B$ is defined by
$$B= \frac{2 \log A}{\log p} -1,$$
then $\alpha $ is transcendental.
\end{thm}

As a remark, a sequence $(a_n)_{n \geq 0}$ is said to be {\itshape ultimately periodic} if there exist integers $k\geq 0$ and $l\geq 1$ such that $a_{n+l} =a_{n}$ for all $n\geq k$.

For example, the following $p$-adic numbers are transcendental:
\begin{gather}
[0, \overline{p-p^{-1}}^{2\cdot 3^0}, \overline{p^{-1}}^{2\cdot 3^1},\overline{p-p^{-1}}^{2\cdot 3^2}, \overline{p^{-1}}^{2\cdot 3^3}, \ldots , \overline{p-p^{-1}}^{2\cdot 3^{2m}}, \overline{p^{-1}}^{2\cdot 3^{2m+1}}, \ldots ], \label{Ex1}
\end{gather}
\begin{align}
[0, \overline{p^{-1}, p^{-2}}^{8\cdot 17^0}, \overline{p-p^{-1},} & \overline{p-p^{-1}}^{8\cdot 17^1}, \overline{p^{-1}, p^{-2}}^{8\cdot 17^2}, \overline{p-p^{-1}, p-p^{-1}}^{8\cdot 17^3} \nonumber \\
& , \ldots , \overline{p^{-1}, p^{-2}}^{8\cdot 17^{2m}}, \overline{p-p^{-1}, p-p^{-1}}^{8\cdot 17^{2m+1}}, \ldots ], \label{Ex2}
\end{align}
where $2\cdot 3^i$ and $8\cdot 17^i$ indicate the number of times a block of partial quotients is repeated.
(\ref{Ex1}) is the case that for $i\geq 0,\ a_{n_{2i}}=p-p^{-1},\ a_{n_{2i-1}}=p^{-1},\ n_i=3^i ,\ \lambda _i=2\cdot 3^i ,\ k_i=1,\ A=p$ in Theorem \ref{Main Thm}.
(\ref{Ex2}) is the case that for $i\geq 0,\ a_{n_{2i}}=p^{-1},\ a_{n_{2i} +1}=p^{-2},\ a_{n_{2i+1}}=a_{n_{2i+1} +1}=p-p^{-1} ,\ n_i=17^i ,\ \lambda _i=8\cdot 17^i ,\ k_i=2,\ A=p^2$ in Theorem \ref{Main Thm}.

A well-known Lagrange's theorem states that the continued fraction expansion for a real number is ultimately periodic if and only if the number is quadratic irrational.
For Schneider continued fractions, $p$-adic analogy of Lagrange's theorem is not true, that is, there exists a quadratic irrational number whose Schneider continued fraction is not ultimately periodic
(See e.g. Weger \cite{Weger}, Tilborghs \cite{Tilborghs}, van der Poorten \cite{Poorten}).
This paper deals with analogy of Lagrange's theorem for Ruban continued fractions.

We prove that  analogy  of Lagrange's theorem for Ruban continued fractions is not true in Section 2.
Auxiliary results for main results are gathered in Section 3.
In Section 4, we prove Theorem \ref{Main Thm} and give criteria of quadratic or transcendental in a certain class of Ruban continued fractions.
These proofs are mainly based on the proof of Baker's results and the non-Archimedean version of Roth's theorem for an algebraic  number field \cite{Silverman}.

\section{Rational and quadratic irrational numbers}

Wang \cite{Wang1} and Laohakosol \cite{Laohakosol} characterized rational numbers with Ruban continued fractions as follows.
\begin{prop} \label{rational}
Let $\alpha $ be a $p$-adic number.
Then $\alpha $ is rational if and only if its Ruban continued fraction expansion is finite or ultimately periodic with the period $p-p^{-1}$.
\end{prop}
\begin{proof}
See \cite{Wang1} or \cite{Laohakosol}.
\end{proof}

Next, we prove that analogy of Lagrange's theorem for Ruban continued fractions is not true by the similar method as in \cite{Weger}.
We consider a Ruban continued fraction for $\alpha = \sqrt{D}$ where $D \in \mathbb{Z}$ not a square, but a quadratic residue modulo $p$, if $p$ is odd, 1 modulo $8$, if $p=2$, so that $\alpha \in \mathbb{Q}_p$.
If the Ruban continued fraction of $\alpha $ is $[a_0 , a_1 , a_2 , \ldots ]$, then there exist rational numbers $R_n , Q_n$ such that
$$\alpha _n = \frac{R_n +\sqrt{D}}{Q_n}$$
for $n \in \mathbb{Z}_{\geq 0}$.
Obviously, $R_0 = 0, Q_0 = 1$, and for all $n$ we have the recursion formula
$$ R_{n+1} = -(R_n - a_n Q_n),\ Q_{n+1} = \frac{D - R_{n+1} ^2}{Q_{n}} $$
by induction on $n$.
\begin{prop} \label{quadratic}
If  $\ R_m Q_m \leq 0,\ \mbox{and}\ R_{m+1}^2 >D\ \mbox{for some}\ m$, then the Ruban continued fraction expansion of $\alpha $ is not ultimately periodic.
\end{prop}
\begin{proof}
We show $R_{m+1} Q_{m+1} < 0,\ R_{m+2}^2 >D,\ \mbox{and}\ |R_{m+2}|>|R_{m+1}|$.
Let us assume $R_m Q_m < 0$.
Then we have $R_m R_{m+1} < 0$ by the recursion formula for $R_{m+1}$.
We also obtain $Q_m Q_{m+1} < 0$ by the recursion formula for $Q_{m+1}$ and $R_{m+1} ^2 > D$.
Hence, we get $R_{m+1} Q_{m+1} <0$.
Furthermore, by $a_{m+1} \not = 0$, we have
$$|R_{m+2}| = |R_{m+1}| + a_{m+1} |Q_{m+1}| > |R_{m+1}|,$$
so that $R_{m+2} ^2 > D$.
Next, let us assume $R_m Q_m = 0$.
By $R_{m} =0$, we have $R_{m+1} = a_m Q_m$.
By the recursion formula for $Q_{m+1}$, we have $Q_{m} Q_{m+1}<0$.
Thus, we obtain $R_{m+1} Q_{m+1} < 0$.
In the same way, we see $|R_{m+2}| > |R_{m+1}|$ and $R_{m+2}^2 >D$.
Since $(|R_n|)_{n \geq m}$ is strictly increasing, the Ruban continued fraction expansion for $\sqrt{D}$ is not ultimately periodic.
\end{proof}

\begin{cor} \label{quadratic cor}
If $D<0$, then the Ruban continued fraction expansion of $p$-adic number $\sqrt{D}$ is not ultimately periodic.
\end{cor}
\begin{proof}
Since $R_0 Q_0 =0$, and $R_1 ^2 \geq 0$, the corollary follows.
\end{proof}

\section{Auxiliary results}
For an infinite Ruban continued fraction $\alpha =[a_0 , a_1 , a_2 , \ldots]$, we define nonnegative rational numbers $q_n , r_n$ by using recurrence equations:
\[
\begin{cases}
q_{-1} =0,\ q_0 =1,\ q_n =a_n q_{n-1} + q_{n-2}, & n\geq 1, \\
r_{-1} =1,\ r_0 =a_0 ,\ r_n =a_n r_{n-1} + r_{n-2}, & n\geq 1.
\end{cases}
\]
Let $\lambda $ be a variable. 
Then the Ruban continued fraction has the following properties which are the same properties as the continued fraction expansions for real numbers: For all $n\geq 0$,
\begin{gather}
[a_0 , a_1 , \ldots ,a_n]=\frac{r_n}{q_n}, \label{pro1} \\
[a_0 , a_1 , \ldots ,a_n , \lambda ]=\frac{\lambda r_n +r_{n-1}}{\lambda q_n +q_{n-1}}, \label{pro2} \\
r_{n-1} q_n - r_n q_{n-1}= (-1)^n . \label{pro3}
\end{gather}
Those are easily seen by induction on $n$.
\begin{lem} \label{fundamental}
The following equalities hold:
\begin{gather}
|q_n|_p =|a_1 a_2 \cdots a_n|_p,\ n\geq 1, \label{pro4} \\
\begin{cases} \label{pro5}
|r_n|_p = |a_0 a_1 \cdots a_n|_p ,\ n\geq 1, & (a_0 \not = 0) \\
|r_1|_p =1,\ |r_n|_p = |a_2 a_3 \cdots a_n|_p ,\ n\geq 2, & (\mbox{otherwise})
\end{cases}
\\
\left|\alpha  -\frac{r_n}{q_n} \right|_p = \frac{1}{|a_{n+1} |_p |q_n|_p ^2},\ n\geq 0. \label{pro6}
\end{gather}
\end{lem}
\begin{proof}
See \cite{Wang1}.
\end{proof}

\begin{lem} \label{upper bound}
If $\alpha '$ is a Ruban continued fraction in which the first $n+1$ partial quotients are the same as those of $\alpha $, then
$$ |\alpha -\alpha '|_p \leq  |q_n|_p ^{-2}.$$
\end{lem}
\begin{proof}
Since $r_n/q_n$ is a $n$-th convergent to both $\alpha $ and $\alpha '$, and (\ref{pro6}), the lemma follows.
\end{proof}
\begin{lem} \label{p-adic value is larger}
The following inequalities hold:
$$q_n \leq |q_n|_p ,\; r_n \leq |r_n|_p ,\ \mbox{for all}\ n\geq -1.$$
\end{lem}
\begin{proof}
The proof is by induction on $n$.
It is obvious that for $n=-1,0$.
By Lemma \ref{fundamental} and the definition of Ruban continued fraction expansions, we have
\begin{eqnarray*}
q_n & \leq  & a_n |q_{n-1}|_p + |q_{n-2}|_p
\leq \left( p- \frac{1}{|a_n|_p} \right) |q_{n-1}|_p + |q_{n-2}|_p \\
& \leq  & |q_{n-1}|_p \left( p+ \frac{1}{p} - \frac{1}{|a_n|_p}\right)
\leq |q_n|_p.
\end{eqnarray*}
The proof for $r_n$ is similar.
\end{proof}

For $\beta \in \overline{\mathbb{Q}}$, let $f_\beta (X)= \sum_{i=0}^{n} d_i X^i$ be a minimum polynomial of $\beta $ in $\mathbb{Z}[X]$.
Put
$$H(\beta ):= \max_{0\leq i \leq n} {|d_i|}.$$
$H(\beta )$ is called a {\itshape primitive height} of $\beta $.
\begin{lem} \label{height estimate}
Suppose $a_0 =0$. Let $h$, $k$ be positive integers and consider the Ruban continued fraction
$$\eta =[0, a_1, \ldots ,a_{h-1}, \overline{a_h, \ldots , a_{h+k-1}}].$$
Then $\eta $ is rational or quadratic irrational. Furthermore, we have
$$ H(\eta )\leq 
\begin{cases}
\ p & (\mbox{if}\ \eta \ \mbox{is rational and } h=1) \\
\ |q_{h-1}|_p ^2 & (\mbox{if}\ \eta \ \mbox{is rational and } h\geq 2) \\
\ 2|q_{h+k-1}|_p ^4 & (\mbox{if}\ \eta \ \mbox{is quadratic irrational}).
\end{cases} $$
\end{lem}
\begin{proof}
By $\eta _h = \eta _{h+k}$ and (\ref{pro2}), we obtain
$$ \eta = \frac{\eta _h r_{h-1} + r_{h-2}}{\eta _h q_{h-1} + q_{h-2}} = \frac{\eta _h r_{h+k-1} + r_{h+k-2}}{\eta _h q_{h+k-1} +q_{h+k-2}}.$$
Eliminating $\eta _h$, we have
$$A\eta ^2 + B\eta + C =0,$$
where
\begin{align*}
A & = q_{h-2}q_{h+k-1} - q_{h-1}q_{h+k-2}, \\
B & = q_{h-1}r_{h+k-2} + r_{h-1}q_{h+k-2} - r_{h-2}q_{h+k-1} - q_{h-2}r_{h+k-1}, \\
C & = r_{h-2}r_{h+k-1} - r_{h-1}r_{h+k-2}.
\end{align*}
Therefore, $\eta $ is either rational or quadratic irrational.
By the assumption $a_0 =0$, it follows that $r_n \leq q_n ,\ |r_n|_p \leq |q_n|_p \ \mbox{for all}\ n\geq 0$.
By induction on $n$, it is easy to check that $r_n|r_n|_p ,\ q_n|q_n|_p \in \mathbb{Z}\ \mbox{for all}\ n\geq 0$.

Let us assume that $\eta $ is a quadratic irrational.
By $|q_{h+k-1}|^2 _p A$, $|q_{h+k-1}|^2 _p B$, $|q_{h+k-1}|^2 _p C \in \mathbb{Z}$ and Lemma \ref{p-adic value is larger}, we obtain
\begin{eqnarray*}
H(\eta ) & \leq & |q_{h+k-1}|_p ^2 \max (|A|,|B|,|C|) \\
& \leq & 2 q_{h+k-1}^2 |q_{h+k-1}|_p ^2
\leq 2|q_{h+k-1}|_p ^4 .
\end{eqnarray*}
Now let us assume that $\eta $ is rational.
By Proposition \ref{rational}, we have 
\begin{eqnarray*}
\eta =[a_0, \ldots ,a_{h-1}, -1/p],
\end{eqnarray*}
that is,
\begin{eqnarray*}
\eta = \frac{p r_{h-2} - r_{h-1}}{p q_{h-2} - q_{h-1}}.
\end{eqnarray*}
When $h=1$, we see that $H(\eta )=p$.
Next we consider the case $h\geq 2$.
Since $(p r_{h-2} - r_{h-1})|q_{h-1}|_p$ and $(p q_{h-2} - q_{h-1})|q_{h-1}|_p$ are integers, we have
\begin{eqnarray*}
H(\eta ) & \leq & \max (|p r_{h-2} - r_{h-1}||q_{h-1}|_p,\ |p q_{h-2} - q_{h-1}||q_{h-1}|_p) \\
& \leq & |q_{h-1}|_p ^2 ,
\end{eqnarray*}
and the lemma follows.
\end{proof}

We recall a height of algebraic numbers which is different from the primitive height.
Let $K$ be an algebraic number field and $\mathcal{O}_K$ be the integer ring of $K$, and $M(K)$ be the set of places of $K$.
For $x \in K$ and $v \in M(K)$, we define the absolute value $|x|_v$ by
\begin{description}
\item[(i)] $|x|_v = |\sigma (x)|$   if $v$ corresponds the embedding $\sigma : K \hookrightarrow \mathbb{R}$
\item[(ii)] $|x|_v = |\sigma (x)|^2 = |\overline{\sigma} (x)|^2$   if $v$ corresponds the pair of conjugate embeddings $\sigma , \overline{\sigma} : K \hookrightarrow \mathbb{C}$
\item[(iii)] $|x|_v = ( \textnormal{N} (\mathfrak{p}))^{- \textnormal{ord} _{\mathfrak{p}} (x)}$   if $v$ corresponds to the prime ideal $\mathfrak{p}$ of $\mathcal{O}_K$.
\end{description}
Set
$$\overline{H}_K (\beta ):= \prod_{v \in M(K)} \max \left\{1, |\beta |_v \right\}$$
for $\beta \in K$.
$\overline{H}_K (\beta )$ is called an {\itshape absolute height} of $\beta $.
Then there are the following relations between primitive and absolute height.
\begin{prop} \label{primitive absolute}
For $b \in \mathbb{Q}$ and $\beta \in \overline{\mathbb{Q}}$ with $[\mathbb{Q}(\beta ) , \mathbb{Q}] = D$, we have
\begin{gather*}
H(b)= \overline{H}_{\mathbb{Q}} (b), \\
\overline{H}_{\mathbb{Q} (\beta )} (\beta ) \leq (D+1)^{1/2} H(\beta ),\ H(\beta ) \leq 2^D \overline{H}_{\mathbb{Q} (\beta )} (\beta ).
\end{gather*}
\end{prop}
\begin{proof}
See Part B of \cite{Silverman}.
\end{proof}

The main tool for the proof of main results is the non-Archimedean version of Roth's theorem for algebraic number fields.
\begin{thm} \label{Roth}
{\normalfont (Roth Theorem).} 
Let $K$ be an algebraic number field, and $v$ be in $M(K)$ with it extended in some way to $\overline{K}$.
Let $\beta \in \overline{K} \backslash K$ and $\delta ,C >0$ be given.
Then there are only finite many $\gamma  \in K$ with the solution of the following inequality:
$$|\beta -\gamma  |_v \leq \frac{C}{\overline{H}_K (\gamma  )^{2+\delta }}.$$
\end{thm}
\begin{proof}
See Part D of \cite{Silverman}.
\end{proof}

\section{Main results}

\begin{proof}[Proof of Theorem \ref{Main Thm}]
We may assume that $a_0=0$.
By the assumption, there are infinitely many positive integers $j$ which satisfy
\begin{eqnarray}
a_{n_j} = a_{n_j +1} = \cdots = a_{n_j +k_j -1} = p-p^{-1} . \label{Lambda1}
\end{eqnarray}
Let $\Lambda $ be an infinite set of $j$ which satisfy (\ref{Lambda1}).

For $i\in \Lambda $, we put
\begin{eqnarray*}
\eta ^{(i)}:=[0,a_1,\ldots ,a_{n_i-1},\overline{p-p^{-1}}].
\end{eqnarray*}
By Proposition \ref{rational}, $\alpha $ is not rational.
Suppose that $\alpha $ is an algebraic number of degree at least two.
We show that if $\chi >2$, then we have
\begin{eqnarray}
|\alpha - \eta ^{(i)}|_p > |q_{n_i -1}|_p ^{-2\chi } \label{contra1}
\end{eqnarray}
for all sufficiently large $i \in \Lambda $.
Suppose the claim is false.
By Proposition \ref{rational}, $\eta ^{(i)}$ is rational for each $i \in \Lambda $.
By Lemma \ref{height estimate} and Proposition \ref{primitive absolute}, we have
\begin{eqnarray*}
|\alpha - \eta ^{(i)}|_p 
& \leq & |q_{n_i -1}|_p ^{-2\chi }
\leq \overline{H} _{\mathbb{Q}} (\eta ^{(i)})^{-\chi } 
\end{eqnarray*}
for infinitely many $i$, which contradicts Theorem \ref{Roth}.

By Lemma \ref{upper bound}, we obtain $|\alpha - \eta ^{(i)}|_p \leq  |q_{m_i}|_p ^{-2}$ for $i \in \Lambda $, where $m_i =n_i +k_i \lambda _i -1$.
Therefore, we get
\begin{eqnarray*}
|q_{m_i}|_p < |q_{n_i -1}|_p ^{\chi }
\end{eqnarray*}
for sufficiently large $i \in \Lambda $.
By Lemma \ref{fundamental}, we see $p^i \leq |q_i|_p \leq A^i$ for $i\geq 1$.
Thus, for all sufficiently large $i \in \Lambda $, it follows that
$$\frac{\lambda _i}{n_i} < B + (\chi -2) \frac{\log A}{\log p} .$$
Since there exists $\delta >0$ such that $\lambda _i >(B +\delta )n_i$ for all sufficiently large $i$, we have for all sufficiently large $i \in \Lambda $,
$$2+ \frac{\log p}{\log A} \delta <\chi .$$
This inequality holds for each $\chi >2$, a contradiction.
\end{proof}

We also obtain the following results.
\begin{thm} \label{Thm2}
Let $\alpha $ be a quasi-periodic Ruban continued fraction, and $A\geq p$ be a real number.
Assume that $(a_i)_{i\geq 0}$ is a non-ultimately periodic sequence such that $|a_i |_p \leq A$ for each $i$, and $(k_i)_{i\geq 0}$ is bounded.
If
$$\limsup_{i\rightarrow \infty } \frac{\lambda _i}{n_i} > B'=B'(A),$$
where $B'$ is defined by
$$B'= \frac{4 \log A}{\log p} -1,$$
then $\alpha $ is quadratic irrational or transcendental.
\end{thm}
\begin{thm} \label{Thm3}
Consider a quasi-periodic Ruban continued fraction
\begin{eqnarray*}
\alpha =[a_0, \ldots , a_{n_0-1}, \overline{a_{n_0}, \ldots , a_{n_0+k_0-1}}^{\lambda _0}, \overline{a_{n_1}, \ldots , a_{n_1+k_1-1}}^{\lambda _1}, \ldots ],
\end{eqnarray*}
where the notation means that $n_i = n_{i-1} + \lambda _{i-1}k_{i-1}$.
Assume that $(a_i)_{i\geq 0}$ is not an ultimately periodic sequence, the sequences $(|a_i|_p)_{i\geq 0} \ \mbox{and}\ (k_i)_{i\geq 0}$ are bounded, and that
\begin{eqnarray*}
\liminf_{i \rightarrow \infty } \frac{\lambda _i}{\lambda _{i-1}} > 4.
\end{eqnarray*}
Then $\alpha $ is quadratic irrational or transcendental.
\end{thm}

\begin{rem}
{\normalfont
There exist quadratic irrational numbers whose Ruban continued fraction expansions are not ultimately periodic  by Corollary \ref{quadratic cor}.
Therefore, it is difficult to determine whether a given Ruban continued fraction is quadratic irrational or transcendental.
However, we see that there exist a transcendental number in the set of Ruban continued fractions which satisfy the assumption of Theorem \ref{Thm2} and \ref{Thm3}.
For example, (\ref{Ex2}) satisfies the assumption of Theorem \ref{Thm2} and \ref{Thm3}.}
\end{rem}

In the  following, $c_1 , c_2 , \ldots ,c_6$ denote positive real numbers which depend only on $\alpha $, and we may assume that $a_0 =0$.

\begin{proof}[Proof of Theorem \ref{Thm2}]
By the assumption, there exists $\delta >0$ such that $\lambda _i >(B'+\delta )n_i$ for infinitely many $i$.
For each positive integer $i$, there are only finitely many possibilities for $k_i$ and for 
$$a_{n_i}, a_{n_i +1}, \ldots , a_{n_i +k_i -1}.$$
Therefore, there exist a positive integer $k$ and $b_1 ,b_2 , \ldots ,b_k \in S' _p$ such that there are infinitely many $j$ which satisfy
\begin{eqnarray}
k_j =k,\ a_{n_j} =b_1 , \ldots ,a_{n_j +k_j -1} =b_k,\ \lambda _j >(B'+\delta )n_j . \label{Lambda2}
\end{eqnarray}
Let $\Lambda $ be an infinite set of $j$ which satisfy (\ref{Lambda2}).

For $i\in \Lambda$, we put
\begin{eqnarray*}
\eta ^{(i)}:=[0,a_1,\ldots ,a_{n_i-1},\overline{b_1,\ldots ,b_k}].
\end{eqnarray*}
By Proposition \ref{rational}, $\alpha $ is not rational.
Suppose that $\alpha $ is an algebraic number of degree at least three.
We show that if $\chi >2$, then we have
\begin{eqnarray}
|\alpha - \eta ^{(i)}|_p > |q_{n_i + k_i -1}|_p ^{-4\chi } \label{contra2}
\end{eqnarray}
for all sufficiently large $i \in \Lambda $.
Suppose the claim is false.
By Lemma \ref{height estimate}, $\eta ^{(i)}$ is rational or quadratic irrational for each $i \in \Lambda $.
Let us assume that $\eta ^{(i)}$ is quadratic irrational.
Then there exists a quadratic field $K$ such that $\eta ^{(i)} \in K$ for all $i \in \Lambda $.
Take a real number $\varepsilon $ which satisfies $0< \varepsilon < \chi -2$.
Then we have $2^{\chi -\varepsilon } < |q_{n_i + k_i -1}|_p ^{4\varepsilon }$ for all sufficiently large $i \in \Lambda $.
Put $v \in M(K)$ with $v \mid p$.
We denote again by $v$ one of the place extended to $K(\alpha )$.
By $[K(\alpha )_v : \mathbb{Q}_p]=1$, Lemma \ref{height estimate}, and Proposition \ref{primitive absolute}, we obtain
\begin{eqnarray*}
|\alpha - \eta ^{(i)}|_v 
& = & |\alpha - \eta ^{(i)}|_p
\leq |q_{n_i + k_i -1}|_p ^{-4\chi } \\
& \leq  & (2|q_{n_i + k_i -1}|_p ^4)^{\varepsilon -\chi }
\leq H(\eta ^{(i)})^{\varepsilon -\chi  } \\
& \leq & \frac{c_1}{\overline{H}_K (\eta ^{(i)})^{\chi -\varepsilon }}
\end{eqnarray*}
for infinitely many $i$, which contradicts Theorem \ref{Roth}.
In the same way, we see (\ref{contra2}) in the case that $\eta ^{(i)}$ is rational.
By Lemma \ref{upper bound}, we have $|\alpha - \eta ^{(i)}|_p \leq  |q_{m_i}|_p ^{-2}$ for $i \in \Lambda $, where $m_i =n_i +k\lambda _i -1$.
Therefore, we obtain
\begin{eqnarray*}
|q_{m_i}|_p < |q_{n_i + k-1}|_p ^{2\chi }
\end{eqnarray*}
for sufficiently large $i \in \Lambda $.
By Lemma \ref{fundamental},  we see $p^i \leq |q_i|_p \leq A^i$ for $i\geq 1$.
Thus, for all sufficiently large $i \in \Lambda $, we have
$$\lambda _i < c_2 +\left( \frac{1}{2} (B'+1)\chi -1 \right), $$
so
$$\left( 1- \frac{\chi }{2} \right) (B'+1) +\delta < \frac{c_2}{n_i}.$$
This inequality holds for each $\chi >2$, and contradicts if $i$ is sufficiently large in $\Lambda $.
\end{proof}

\begin{proof}[Proof of Theorem \ref{Thm3}]
Put
$$P_h ^{(i)} := [a_{n_i +h-1},a_{n_i +h-2}, \ldots ,a_{n_i}, \overline{a_{n_i +k_i -1},a_{n_i +k_i -2}, \ldots ,a_{n_i}}]$$
for $i=0,1,2,\ldots \ \mbox{and}\ h=1,2,\ldots ,k_i$.
Put
$$ P^{(i)} := \prod_{h=1}^{k_i} P_h ^{(i)} .$$
For each positive integer $i$, there exist only finitely many possibilities for $k_i$ and for 
$$a_{n_i}, a_{n_i +1}, \ldots , a_{n_i +k_i -1}.$$ 
$P^{(i)}$ is a function which depends only on  $k_i , a_{n_i}, a_{n_i +1}, \ldots , a_{n_i +k_i -1}$.
Hence, there exists a real number $P$ such that the greatest of those values $|P^{(i)}|_p$ which are attained for infinitely many $i$.
Then there exists an integer $l$ such that
$$|P^{(i)} |_p \leq P\ \mbox{for all}\ i\geq l.$$
There exist a positive integer $k$ and $b_1 ,b_2 , \ldots ,b_k \in S' _p$ such that there are infinitely many $j$ which satisfy
\begin{eqnarray}
 |P^{(j)}|_p = P,\ k_j =k,\ a_{n_j} =b_1 , \ldots ,a_{n_j +k_j -1} =b_k . \label{Lambda3}
\end{eqnarray}
Let $\Lambda $ be an infinite set of $j$ which satisfy (\ref{Lambda3}).
We may assume that $l=0$.

Let us show that
\begin{gather}
|q_{n_{i+1} -1}|_p \leq c_3 P^{\lambda _i} |q_{n_i -1}|_p\ \ \ \mbox{for all}\ i, \label{all i} \\
|q_{n_{i+1} -1}|_p \geq c_4 P^{\lambda _i} |q_{n_i -1}|_p\ \ \ \mbox{for all}\ i \in \Lambda . \label{all i in Lambda}
\end{gather}
Firstly, an induction allows us to establish the mirror formula
$$ \frac{q_m}{q_{m-1}} = [a_m, \ldots ,a_1], \ \mbox{for all}\ m\geq 1.$$
Put
$$ W_h ^{(i)} := \frac{q_{n_i +h-1}}{q_{n_i +h-2}} ,$$
for $i=0,1,2, \ldots \ \mbox{and}\ h=1,2, \ldots ,k_i \lambda _i$, and
$$ W^{(i)} := \prod_{h=1}^{k_i \lambda _i} W_h ^{(i)} .$$
Clearly, we have $q_{n_{i+1} -1} =W^{(i)} q_{n_i -1}$.
It follows from Lemma \ref{fundamental} and \ref{upper bound} that for any $i$,
\begin{eqnarray*}
|W^{(i)}|_p & = & \prod_{h=1}^{k_i} \prod_{s=0}^{\lambda _i -1} |W_{h+s k_i} ^{(i)}|_p
\leq \prod_{h=1}^{k_i} \prod_{s=0}^{\lambda _i -1} (|P_h ^{(i)}|_p + |U_{h,s} ^{(i)}|_p ^{-2}) \\
& \leq & \prod_{h=1}^{k_i} \prod_{s=0}^{\lambda _i -1} |P_h ^{(i)}|_p (1+p^{-2(h+s k_i -1)})
\leq |P^{(i)}|_p ^{\lambda _i} \prod_{h=1}^{k_i \lambda _i} (1+p^{2-2 h}) \\
& \leq & c_3 P^{\lambda _i} ,
\end{eqnarray*}
where $U_{1,0} ^{(i)} =1$ and otherwise $U_{h,s} ^{(i)}$ is the denominator of $(h+s k_i -1)$-th convergent to $P_{h} ^{(i)}$.
Likewise, for all $i$, we have
\begin{eqnarray*}
|W^{(i)}|_p & \geq & \prod_{h=1}^{k_i} \prod_{s=0}^{\lambda _i -1} (|P_h ^{(i)}|_p - |U_{h,s} ^{(i)}|_p ^{-2}) \\
& \geq & |P^{(i)}|_p ^{\lambda _i} \left( 1- \frac{1}{|P_1 ^{(i)}|_p}\right) \prod_{h=2}^{k_i \lambda _i} (1-p^{2-2 h}).
\end{eqnarray*}
If $i \in \Lambda $, then $|P^{(i)}|_p = P$ and $P_1 ^{(i)}$ is independent of $i$. 
Therefore, we obtain
$$ |W^{(i)}|_p \geq c_4 P^{\lambda _i} \ \ \ \mbox{for all}\ i \in \Lambda .$$
If $A$ and $K$ are the upper bounds of the sequences $(|a_i|_p)_{i\geq 0}$ and $(k_i)_{i\geq 0}$, then for all $i$, we have
\begin{eqnarray}
|q_{n_i +k_i -1}|_p \leq A^K |q_{n_i -1}|_p . \label{trivial}
\end{eqnarray}

Now, there exist a real number $\delta >0$ and an integer $N\geq 1$ such that $\lambda _i >(4+\delta )\lambda _{i-1}$ for all $i>N$.
Set $\chi :=2 + \delta /4$.
For $i\in \Lambda$, we put
\begin{eqnarray*}
\eta ^{(i)}:=[0,a_1,\ldots ,a_{n_i-1},\overline{b_1,\ldots ,b_k}].
\end{eqnarray*}
By Proposition \ref{rational}, $\alpha $ is not rational.
Suppose that $\alpha $ is an algebraic number of degree at least three.
Then we have
\begin{eqnarray*}
|\alpha - \eta ^{(i)}|_p > |q_{n_i + k_i -1}|_p ^{-4\chi }
\end{eqnarray*}
for all sufficiently large $i \in \Lambda $.
This follows by the same way as in the proof of Theorem \ref{Thm2}.
By Lemma \ref{upper bound}, we see $|\alpha - \eta ^{(i)}|_p \leq  |q_{n_{i+1} -1}|_p ^{-2}$ for all $i$.
Therefore, we obtain
\begin{eqnarray}
|q_{n_{i+1} -1}|_p < |q_{n_i + k_i -1}|_p ^{2\chi } \label{contra3}
\end{eqnarray}
for all sufficiently large $i \in \Lambda $.
Applying (\ref{all i}), (\ref{all i in Lambda}), (\ref{trivial}), and (\ref{contra3}), we have for all sufficiently large $i \in \Lambda $,
$$P^{\lambda _i} < c_5 c_6 ^i P^{(2\chi -1)(\lambda _{i-1} + \lambda _{i-2} + \cdots + \lambda _N)} .$$
Taking logarithms, we see that for all sufficiently large $i \in \Lambda $,
\begin{eqnarray*}
\frac{\lambda _i}{\lambda _{i-1}} & < & \frac{\log c_5 +i \log c_6}{\lambda _{i-1} \log P} + (2\chi -1) \sum_{j=0}^{\infty } \left( \frac{1}{4+\delta /2} \right)^j \\
& = & \frac{\log c_5 +i \log c_6}{\lambda _{i-1} \log P} +4+ \frac{\delta }{2} .
\end{eqnarray*}
Since $i/\lambda _i \rightarrow 0\ \mbox{as}\ i \rightarrow \infty $, we have
$$\frac{\lambda _i}{\lambda _{i-1}} < \frac{\delta }{2} +4+ \frac{\delta }{2} = 4+\delta $$
for all sufficiently large $i \in \Lambda $.
This contradicts, and the proof is complete.
\end{proof}

\subsection*{Acknowledgements}
I am greatly indebted to Prof.\ Kenichiro Kimura for several helpful comments concerning the proof of main theorems.
I also wish to express my gratitude to Prof.\ Masahiko Miyamoto.

\end{document}